\documentclass[11pt,a4paper,leqno,twoside,notitlepage]{article}

\usepackage{amsmath}
\usepackage{amsfonts}
\usepackage{amssymb}
\usepackage{amsxtra}
\usepackage{amsthm}
\usepackage{enumerate}
\usepackage{ifthen}
\usepackage{xypic}

\newboolean{details}

\theoremstyle{plain}
\newtheorem{thm}[equation]{Theorem}
\newtheorem*{thm*}{Theorem}
\newtheorem{prop}[equation]{Proposition} 
\newtheorem*{prop*}{Proposition}
\newtheorem{lem}[equation]{Lemma}

\theoremstyle{definition}
\newtheorem{defn}[equation]{Definition}

\newcommand{\nb}{neighborhood}

\newcommand{\mC}{{\mathbb C}}

\newcommand{\mN}{{\mathbb N}}

\newcommand{\mR}{{\mathbb R}}

\newcommand{\mcA}{\mathcal{A}}
\newcommand{\mcB}{\mathcal{B}}
\newcommand{\mcC}{\mathcal{C}}

\newcommand{\frakH}{\mathfrak{H}}

\newcommand{\CN}{C^{\ast}(N)}

\newcommand{\Cst}{C^{\ast}}

\newcommand{\twoheadlongrightarrow}{\,\,\rightarrow\mspace{-24.0mu}\longrightarrow\,}

\renewcommand{\to}{\longrightarrow}

\newcommand{\onto}{\mbox{$\twoheadlongrightarrow$}}

\newcommand{\mdot}{\!\cdot\!}
\newcommand{\clos}{^{\textbf{--}}}

\providecommand{\Id}{\mathop{\rm Id}\nolimits}

\providecommand{\ext}{\mathop{\rm ext}\nolimits}

\providecommand{\Prim}{\mathop{\rm Prim}\nolimits}

\providecommand{\res}{\mathop{\rm res}\nolimits}

\setboolean{details}{false}

\begin{document}
\thispagestyle{empty}
\boldmath
\centerline{\Large\bfseries{On Glimm's Theorem for locally
quasi-compact}}
\vspace{0.5cm}
\centerline{\Large\bfseries almost Hausdorff $G$-spaces}
\unboldmath
\vspace{1.5cm}
\centerline{\large Oliver Ungermann}\vspace{0.5cm}

\vspace{1cm}
\centerline{\large July 2016}
\vspace{1cm}
\hspace{7mm}\parbox{110mm}
{\small\noindent\textbf{Abstract.} 
We establish a characterization of the well-behaved orbits of
a totally Baire $G$-space of a hereditary Lindel\"of locally 
compact group under a mild assumption of Hausdorffness.
Then we give a reformulation of Glimm's theorem generalizing the 
assertion from second countable to hereditary Lindel\"of locally
compact groups acting on second countable spaces.
Next we show that almost Hausdorff $G$-spaces of compact groups 
are almost proper and hence regular.
Finally we recall some applications of these results in 
harmonic analysis.\\[1cm]
\centerline{\textbf{MSC 2000: }22-01, 22F05, 22F30}}

\vspace{0.5cm}

\subsection*{Introduction}
In 1961 James Glimm proved the following theorem.
\begin{thm*}
 Let $G$ be a second countable locally compact group and $X$ a second 
 countable locally quasi-compact $G$-space. Suppose that every closed
 subset of $X$ contains a non-empty relatively open Hausdorff subset.
 Then the following are equivalent:
 \begin{enumerate}
  \item Each orbit in $X$ is relatively open in its closure.
  \item The orbit space $G\setminus X$ is $T_0$.
  \item $G\setminus X$ is countably separated.
  \item For each quasi-invariant ergodic Borel measure $\beta$, there
  is an orbit $G\mdot x$ in $X$ such that $\beta(X\setminus G\mdot x)=0$.
  \item Every closed subset of $G\setminus X$ contains a non-empty
  relatively open Hausdorff subset.
  \item For each $x\in X$, the map $gG_x\mapsto g\mdot x$ from $G/G_x$
  onto $G\mdot x$ is a homeomorphism, where $G\mdot x$ has the relative
  topology as a subspace of $X$.
  \item For each neighborhood $N$ of $e$ (the identity of $G$), each 
  non-empty locally quasi-compact $G$-invariant subset $V$ of $X$ and
  each non-empty relatively open subset $V_0$ of $V$ there is a non-empty
  relatively open subset $U$ of $V_0$ such that for each $x\in U$, 
  $N\mdot x\cap U=G\mdot x\cap U$.
 \end{enumerate}
\end{thm*}

\noindent In this note we give a reformulation of the proof of this
theorem taking into account some addenda of M.\ Rieffel and D.\ P.\ Williams.
We emphasize topological aspects of the theory and avoid measure theory.
We generalize the assertion of Glimm's theorem from second countable
groups to hereditary Lindel\"of locally compact groups $G$ acting on second
countable almost Hausdorff spaces $X$.\\\\
As an application of Glimm's theorem we state variants of the following
result which constitutes a first step in Mackey's theory of little groups.
Let $N$ be a second countable closed normal subgroup of a hereditary 
Lindel\"of locally compact group $G$. Since $G$ acts on $N$ and hence
on $\CN$ by conjugation, the unitary dual $X=\widehat{N}$ of $N$ becomes
a $G$-space in a natural way. Suppose that $N$ is of type I and that all
orbits of $\widehat{N}$ are locally closed. Then for every factor
representation $\pi$ of $G$ there exists a unique orbit $G\mdot\sigma$
of $\widehat{N}$ such that $\pi\,|\,N$ is weakly equivalent
to $G\mdot\sigma$.

\subsection*{Transformation groups}
\noindent In this section we shall discuss group actions on topological spaces
which do not satisfy the Hausdorff separability axiom. Throughout this
text all topological groups $G$ and topological spaces $X$ are always
understood to be $T_0$: If $x\neq y$ are in $X$ such that $x\in\{y\}\clos$,
then there exists an open $x$-\nb\ $U$ of $X$ such that $y\not\in U$.\\\\
We begin with the basic definitions and some elementary results.

\begin{lem}
If $G$ is a topological group (which is $T_0$), then $G$ is Hausdorff.
\end{lem}
\begin{proof}
Let $g_1\neq g_2$ be in $G$. Since $G$ is $T_0$, we can assume without
loss of generality that $g_2\not\in\{g_1\}\clos$. Then there exists a symmetric
open $e$-\nb\ $U$ of $G$ such that $g_1\not\in Ug_2$. Choose a symmetric
open $e$-\nb\ $W$ of $G$ such that $W^2\subset U$. Now it is easy to
see that $Wg_1$ and $Wg_2$ are disjoint open \nb s of $g_1$ and $g_2$
respectively.
\end{proof}

\noindent A $G$-space $X$ of a topological group $G$ consists of a
($T_0$-)space $X$ and a continuous map $G\times X\to X$ satisfying
$g\mdot (h\mdot x)=(gh)\mdot x$ and $e\mdot x=x$ for all $g,h\in G$
and $x\in X$. A $G$-space is called transitive if $X=G\mdot x$ for one
and hence for all $x\in G$. It is well-known that stabilizers are
closed, see e.g.~ Lemma~1 of~\cite{Blatt1}.

\begin{lem}[Robert J.~Blattner, 1965]\hfill\\
\label{Glimm_lem:closed_stabilizer}
Let $X$ be a $G$-space and $x\in X$ arbitrary. Then $G_x=\{g\in G:g\mdot x=x\}$
is a closed subgroup of $G$.
\end{lem}
\begin{proof}
Obviously $G_x$ is a subgroup. Let $g_\lambda$ be a net in $G_x$ which
converges to $g\in G$. Then $x=g_\lambda\mdot x\to g\mdot x$ so that
$g\mdot x\in\{x\}\clos$. On the other hand, $g\mdot x=g\mdot g_\lambda^{-1}
\mdot x\to x$ so that $x\in\{g\mdot x\}\clos$. Since $X$ is $T_0$, it
follows $g\mdot x=x$. Thus $G_x$ is closed.
\end{proof}

\begin{lem}\label{Glimm_lem:GmodH_is_Hausdorff}
If $H$ is a closed subgroup of a topological group $G$, then $G/H$ is
Hausdorff and the quotient map $\pi:G\onto G/H$ is open.
\end{lem}
\begin{proof}
Obviously $\pi$ is open. Let $g_1,g_2\in G$ such that $g_1H\neq g_2H$.
As $g_2H$ is closed and $g_1\not\in g_2H$, there is an open $e$-\nb\ $U$
of $G$ such that $Ug_1\cap g_2H=\emptyset$. Choose a symmetric open
$e$-\nb\ $W$ of $G$ such that $W^2\subset U$. Now it is easy to see that
$Wg_1H\cap Wg_2H=\emptyset$ so that $\pi(Wg_1)$ and $\pi(Wg_2)$ are
disjoint open \nb s of $g_1H$ and $g_2H$ in $G/H$ respectively.
\end{proof}

\noindent A topological space $X$ is quasi-compact if every open
cover of $X$ admits a finite subcover. We say that $X$ is locally
quasi-compact if every $x$-\nb\ of $X$ contains a quasi-compact
$x$-\nb. Further $X$ is (locally) compact if $X$ is (locally)
quasi-compact and Hausdorff. If $H$ is a closed subgroup of a
(locally) compact group $G$, then $G/H$ is (locally)
compact, too.\\\\
A subset $A$ of a topological space $X$ is said to be locally closed if
$A$ is the intersection of a closed and an open subset of $X$.
This is the case if and only if $A$ is relatively open in its
closure, or equivalently, if every $x\in A$ has an open \nb\ $U$ such
that $U\cap\overline{A}\subset A$.

\begin{lem}
 Let $H$ be a subgroup of a topological group $G$. Then $H$ is locally
 compact in the relative topology of $G$ if and only if $H$ is closed in~$G$.
\end{lem}
\begin{proof}
 Suppose that $H$ is locally compact. Let $g\in H$ be arbitrary. Let $U$ be an
 open $g$-\nb\ of $G$ such that the closure of $U\cap H$ in $H$ is compact. Since
 $G$ is Hausdorff, it follows that the closure of $U\cap H$ in $G$ is contained
 in $H$. This shows $U\cap\bar{H}\subset H$. This means that $H$ is a locally
 closed subset of $G$, i.e., $H$ is an open subgroup of its closure $\bar{H}$
 in $G$. This implies that $H$ is a closed subgroup of $G$. The
 converse is obvious.
\end{proof}

\noindent It is necessary to distinguish carefully between transitive
$G$-spaces and homogeneous spaces.

\begin{defn}\label{Glimm_defn:homogeneous_space}
Let $G$ be a topological group and $X$ a transitive $G$-space. We
say that $X$ is a homogeneous space if $\omega_x:G/G_x\to X$,
$\omega_x(gG_x)=g\mdot x$ is a homeomorphism for one and
hence for all $x\in X$.
\end{defn}

\noindent Lemma~\ref{Glimm_lem:GmodH_is_Hausdorff} implies that
every homogeneous space is Hausdorff.\\\\
Let $X$ be a $G$-space and $x\in X$. The orbit $G\mdot x$ is
a transitive $G$-space. By definition $G\mdot x$ is a homogeneous
space if and only if $\omega_x:G/G_x\to G\mdot x$
is a homeomorphism with respect to the relative topology of
$X$ on $G\mdot x$.

\begin{lem}\label{Glimm_defn:characterize_homogeneous_spaces}
Let $X$ be a transitive $G$-space of a topological group $G$. Then $X$
is a homogeneous space if and only if $N\mdot x$ is an $x$-neighborhood
of $X$  for every $e$-neighbor\-hood $N$ of $G$ and every $x\in X$.
\end{lem}

\noindent In the sequel we will investigate whether the orbit space $G\setminus X$
of a $G$-space $X$ has nice topological properties. In particular we
shall be concerned with the question if its orbits are locally closed
in~$X$ and locally compact in the relative topology of $X$.

\subsection*{Baire spaces and the open mapping theorem}
Our aim is to prove two variants of the open mapping theorem. To this end 
we recall basic facts about Baire spaces and almost Hausdorff spaces.

\begin{defn}\label{Glimm_defn:Baire_spaces}
A topological space $X$ is called a Baire space if the intersection
$\bigcap_{n=1}^\infty U_n$ of any sequence of dense open subsets $U_n$
of $X$ is dense in $X$.
\end{defn}

\noindent It is easy to see that $X$ is a Baire space if and only if
the union $\bigcup_{n=1}^\infty A_n$ of any sequence $A_n$ of closed
subsets with empty interior has empty interior. This contraposition
of the assertion in Definition~\ref{Glimm_defn:Baire_spaces} is used
frequently.

\begin{defn}
 A topological space $X$ is called totally Baire if every closed subspace
 of~$X$ is Baire.
\end{defn}

\begin{defn}\label{Glimm_defn:almost_Hausdorff}
A topological space $X$ is called almost Hausdorff if every non-empty,
closed subset $A$ of $X$ contains a non-empty, relatively open
Hausdorff subset.
\end{defn}

\noindent A singleton in $X$ is simply a one-point subset $\{x\}$ of $X$.
First properties of almost Hausdorff spaces are

\begin{lem}\label{Glimm_lem:almost_Hausdorff_spaces}
Let $X$ be an almost Hausdorff space.
\begin{enumerate}
\item Singletons are locally closed in $X$. In particular $X$ is $T_0$.
\item \textbf{Every} non-empty subset $B$ of $X$ contains a non-empty,
relatively open Hausdorff subset, and is hence itself almost
Hausdorff.
\item There exists a dense open Hausdorff subset of $X$.
\end{enumerate}
\end{lem}
\begin{proof}\hfill
\begin{enumerate}\renewcommand{\labelenumi}{\textit{\arabic{enumi}}.}
\item Let $x\in X$ be arbitrary. Since $X$ is almost Hausdorff, there
exists an open subset $U$ of $X$ such that $U\cap\{x\}\clos$ is non-empty
and Hausdorff. Obviously $x\in U$ because $U$ is open, and
$U\cap\{x\}\clos=\{x\}$ because $U\cap\{x\}\clos$ is Hausdorff. Thus
$\{x\}$ is locally closed.\\\\
Let $x,y\in X$ be arbitrary. If $y\not\in\{x\}\clos$, then we are
done because $X\setminus\{x\}\clos$ is an open \nb\ of $y$ not
containing $x$. So we can assume $y\in\{x\}\clos$. Since $\{x\}$ is
locally closed, there exists an open $x$-\nb\ of $X$ such that
$U\cap\{x\}\clos=\{x\}$ does not contain $y$. Hence $X$ is $T_0$.
\item Let $B$ be an arbitrary subset of $X$ and $A$ its closure in $X$.
Since $X$ is almost Hausdorff, there is a non-empty, relatively open
Hausdorff subset $U$ of $A$. Then $U\cap B$ is a  non-empty,
relatively open Hausdorff subset of $B$.
\item Zorn's Lemma shows that there exists a maximal non-empty, open
Hausdorff subset $U$ of $X$. Suppose that $U$ is not dense in $X$.
Then there exists a non-empty, open Hausdorff subset $V$ of
$X\setminus\overline{U}$. Now $U\cup V$ is open in $X$ and
Hausdorff, in contradiction to the maximality of $U$.
This proves our claim.
\end{enumerate}
\end{proof}

\noindent A $G_\delta$-subset of a topological space is a countable intersection
of open subsets.

\begin{lem}\label{Glimm_lem:locally_compact_implies_Baire}
If $W$ is a $G_\delta$-subset of a locally compact space $X$, then $W$
is a Baire space in the relative topology of $X$.
\end{lem}
\begin{proof}
Let $U_n$ be a sequence of open subsets of $X$ such that $U_n\cap W$
is dense in $W$. Let $V$ be an open subset of $X$ such that
$V\cap W\neq\emptyset$. Since $W$ is a $G_\delta$, there exists a
sequence $W_n$ of open subsets of $X$ such that $W=\bigcap_{n=1}^\infty W_n$.
By induction we will define a decreasing sequence of compact subsets
$K_n$ of $X$ such that ${\rm int}(K_n)\cap W\neq\emptyset$ and
$K_n\subset U_n\cap W_n$: Let $K_0\subset V$ be compact such that
${\rm int}(K_0)\cap W\neq\emptyset$. Suppose that $K_0,\ldots,K_{n-1}$
have been chosen. Since ${\rm int}(K_{n-1})\cap W\neq\emptyset$ and $U_n\cap W$
is dense in $W$, it follows that $K_{n-1}\cap U_n\cap W\neq\emptyset$.
As $X$ is locally compact, there exists a compact subset
$K_n\subset K_{n-1}\cap U_n\cap W_n$ such that ${\rm int}(K_n)\cap W
\neq\emptyset$. By the finite intersection property we obtain
$\bigcap_{n=1}^\infty K_n\neq\emptyset$ and hence
$(\bigcap_{n=1}^\infty U_n)\cap V\cap W\neq\emptyset$. For $V$ was
arbitrary, we see that $(\bigcap_{n=1}^\infty U_n)\cap W$ is
dense in $W$.
\end{proof}

\begin{lem}\label{Glimm_lem:Baire_spaces}
Let $X$ be a topological space.
\begin{enumerate}
\item If $X$ is Baire and $U$ is an open subset of $X$, then
$U$ is Baire.
\item If $U$ is a dense open subset of $X$ and if $U$ is Baire,
then $X$ is Baire.
\item If $W$ is a $G_\delta$-subset of a locally quasi-compact, almost
Hausdorff space $X$, then $W$ is Baire.
\item Every dense $G_\delta$-subset $W$ of a Baire space $X$ is Baire.
\end{enumerate}
\end{lem}

\begin{proof}\hfill
\begin{enumerate}\renewcommand{\labelenumi}{\textit{\arabic{enumi}}.}
\item Let $U_n$ be dense and open subsets of $U$. Then $V_n=U_n\cup
(X\setminus\overline{U})$ is dense and open in $X$. Since $X$ is Baire,
it follows that $\bigcap_{n=1}^\infty V_n$ is dense in $X$. Now it clear that
$\bigcap_{n=1}^\infty U_n$ is dense in $U$.
\item Let $U_n$ be dense and open subsets of $X$. Then $U_n\cap U$ is dense
and open in $U$. Since $U$ is Baire, the subset $(\bigcap_{n=1}^\infty U_n)\cap U$
is dense in $U$. It follows that $\bigcap_{n=1}^\infty U_n$ is dense in $X$.
\item Since $X$ is almost Hausdorff, there exists a dense open Hausdorff subset
$U$ of $X$ by part \textit{3.}\ of Lemma~\ref{Glimm_lem:almost_Hausdorff_spaces}.
Note that $U$ is locally compact. Since $U\cap W$ is a $G_\delta$-subset
of $U$, it follows that $U\cap W$ is Baire by Lemma~
\ref{Glimm_lem:locally_compact_implies_Baire}. For $U\cap W$ is dense
in $W$, we see that $W$ is Baire by part~\textit{2.}\ of this lemma.
\item Let $U_n$ be a sequence of open subsets of $X$ such that $U_n\cap W$
is dense in $W$. As $W$ is dense, it follows that $U_n$ is dense in $X$ for
all~$n$. Since $W$ is a $G_\delta$-subset of~$X$, there exists a sequence of
open subsets of~$X$ such that $W=\cap_{n=1}^\infty W_n$. Note that $W_n$ is
dense in $X$ for all $n$. Now the assumption that $X$ is Baire implies that
$(\cap_{n=1}^\infty U_n)\cap W=\cap_{n=1}^\infty(U_n\cap W_n)$
is dense in~$X$ and thus dense in~$W$. This proves that $W$ is Baire.
\end{enumerate}
\end{proof}

\begin{lem}\label{Glimm_lem:almost_Hausdorff_locally_quasi_compact_implies_Baire}
 Every almost Hausdorff locally quasi-compact space $X$ is totally Baire.
\end{lem}
\begin{proof}
Let $A$ be an arbitrary closed subset of $X$. Since $X$ is almost Hausdorff,
there exists a dense open Hausdorff subset $U$ of $A$. Since $X$ is locally
quasi-compact, it follows that $U$ is locally compact. Now
Lemma~\ref{Glimm_lem:locally_compact_implies_Baire} implies that $U$ is
Baire. Since $U$ is dense in $A$, it follows by part \textit{2.} of
Lemma~\ref{Glimm_lem:Baire_spaces} that $A$ is Baire.
\end{proof}

\noindent A Hausdorff topological space is $\sigma$-compact if it is a
countable union of compact subsets.\\\\
The following variants of the open mapping theorem involve Baire spaces.

\begin{thm}[Open mapping theorem I]\label{Glimm_thm:open_mapping_I}\hfill\\
Let $G$ be a locally compact group and $\varphi:X\to Y$ a $G$-equivariant continuous
map of a $\sigma$-compact homogeneous $G$-space $X$ onto a transitive $G$-space 
and Baire Hausdorff space $Y$. Then it follows that $\varphi$ is an open map.
\end{thm}
\begin{proof}
The proof is standard. Let $x\in X$ be arbitrary and $U$ an $x$-\nb. We must prove that
$\varphi(U)$ is a $\varphi(x)$-\nb. Translating $W$ we can assume $\varphi(x)\in W$.
Shrinking $U$ if necessary we can establish $\varphi(U)\subset W$. By continuity there
exists an $e$-\nb\ $N$ of $G$ such that $N\mdot x\subset U$.  Since $G$ is locally
compact, we can choose a compact symmetric $e$-\nb\ $L$ of $G$ such that $L^2\subset N$.\\
Since $X$ is $\sigma$-compact, there is a sequence $g_m\in G$ such that
$X=\cup_{n=0}^\infty\;g_mL\mdot x$. Thus $Y=\cup_{n=0}^\infty\;g_mL\mdot\varphi(x)$
because $\varphi$ is surjective and $G$-equivariant. As $Y$ is Hausdorff, the subsets
$g_mL\mdot\varphi(x)$ are closed in $Y$. Since $Y$ is a Baire space, there exists
$m\in\mN$ such that $g_mL\mdot\varphi(x)$ has non-empty interior. Let $y$ be an
interior point of~$g_mL\mdot\varphi(x)$ and write $y=g_mg\mdot\varphi(x)$ with $g\in L$.
Now it follows that $\varphi(x)$ is an interior point of~$g^{-1}L\mdot\varphi(x)$.
Since $g^{-1}L\mdot\varphi(x)\subset N\mdot\varphi(x) \subset\varphi(U)$, we see
that $\varphi(U)$ is a $\varphi(x)$-\nb. This proves $\varphi$ to be open.
\end{proof}

\begin{thm}[Open mapping theorem II]\label{Glimm_thm:open_mapping_II}\hfill\\
Let $G$ be a second countable locally compact group and $\varphi:X\to Y$ a $G$-equivariant
continuous map of a homogeneous $G$-space $X$ onto a transitive $G$-space and Baire
space~$Y$ which admits a non-empty open Hausdorff subset. Then $\varphi$ is an open
map.
\end{thm}
\begin{proof}
The proof is similar to the previous one. In addition, it uses an idea from~\cite{Will}.
Let $x\in X$ be arbitrary and $U$ an $x$-\nb. We must prove that $\varphi(U)$ is a
$\varphi(x)$-\nb. We can assume $\varphi(x)\in W$ and $\varphi(U)\subset W$. Let 
$N$ be an $e$-\nb\ of $G$ such that $N\mdot x\subset U$ and $L$ a compact symmetric
$e$-\nb\ of $G$ such that $L^2\subset N$.\\
Let $W$ be a non-empty open Hausdorff subset of $Y$. We consider the open subset
$W_0:=\{a\in G:\varphi(a\mdot x)\in W\}$ of $G$. Since $G$ is second countable,
there exist $g_m\in G$ such that $\{g_m:m\in\mN\}$ is dense in~$W_0$ and a
countable neighborhood basis $\{L_n\mdot x:m\in\mN\}$ of~$e$
in~$G$ satisfying $L_n\subset L$ for all $n$. Let $I$ denote the set of all
$(m,n)\in\mN\times\mN$ such that $g_mL_n\subset W_0$. We claim that
$W_0=\cup_{(m,n)\in I}g_mL_n$. Let $a\in W_0$ be arbitrary. There exists $n\in\mN$
such that $aL_n^{-1}L_n\subset W_0$ and $m\in\mN$ such that $g_m\in aL_n^{-1}$.
Now we obtain $a\in g_mL_n$ and $g_mL_n\subset aL_n^{-1}L_n\subset W_0$
proving the claim. This argument is adpapted from the proof of
\textit{(c)}$\Rightarrow$\textit{(d)} of Theorem~6.2 on p.~177 of~\cite{Will}.\\
Applying the maps $G\to X$, $a\mapsto a\mdot x$, and $\varphi$ we get
$W=\cup_{(m,n)\in I}g_mL_n\mdot\varphi(x)$. Note that the sets $g_mL_n\mdot\varphi(x)$
are closed in $W$ with respect to the relative topology. By part~\textit{1.}\ of
Lemma~\ref{Glimm_lem:Baire_spaces} we know that $W$ is a Baire space. Hence
there exists $(m,n)\in I$ such that $g_mL_n\mdot\varphi(x)$ has non-empty interior
in $W$ and consequently a non-empty interior in~$Y$. As in the proof
of Theorem~\ref{Glimm_thm:open_mapping_I} we now conclude that $\varphi(U)$ is a
$\varphi(x)$-\nb\ which proves $\varphi$ to be open.
\end{proof}

\noindent We emphasize that in Theorem~\ref{Glimm_thm:open_mapping_II} we
do not assume $Y$ to be Hausdorff. However, the assumption that $Y$ is a
transitive $G$-space and a Baire space which contains a non-empty open
Hausdorff subset implies that $Y$ is a homogeneous
space and hence Hausdorff.\\\\
A topological space $X$ is called Lindel\"of if every open cover of $X$ has
a countable subcover. A space $X$ is hereditary Lindel\"of if every open subset
of $X$ is Lindel\"of in the relative topology. Let $X$ be a locally compact space.
Then $X$ is hereditary Lindel\"of if and only if every open subset of $X$ can
be exhausted by a countable family of compact subsets. Clearly every second
countable space is hereditary Lindel\"of.

\begin{thm}[Open mapping theorem III]\label{Glimm_thm:open_mapping_III}\hfill\\
 Let $G$ be a locally compact group and $\varphi:X\to Y$  a $G$-equivariant
 continuous map of a hereditary Lindel\"of homogeneous space $X$
 onto a Baire space $Y$ which contains a non-empty open Hausdorff subset.
 Then $\varphi$ is an open map. In particular, $Y$ is a homogeneous space, too.
\end{thm}
\begin{proof}
The proof is similar to the previous one. Let $x\in X$ be arbitrary and $U$ an
$x$-\nb. We must prove that $\varphi(U)$ is a $\varphi(x)$-\nb. We can assume
$\varphi(x)\in W$ and $\varphi(U)\subset W$. Let $N$ be an $e$-\nb\ of $G$
such that $N\mdot x\subset U$ and $L$ a compact symmetric $e$-\nb\ of $G$
such that $L^2\subset N$.\\
Let $W$ be a non-empty open Hausdorff subset of $Y$. We consider the open subset
$W_0:=\varphi^{-1}(W)$ of $X$. Note that the family $\{{\rm int}(gL_0\mdot x):g\in G$
and $L_0$ is a compact $e$-\nb\ of $G$ contained in $L$ such that
$gL_0\mdot x\subset W_0\}$ is an open cover of $W_0$. Since $X$ is hereditary
Lindel\"of, we can find a countable subcover. Thus there exist $g_n\in G$ and
compact $e$-\nb s $L_n\subset L$ such that $W_0=\cup_{n=1}^\infty g_nL_n\mdot x$.\\
Applying $\varphi$ we get $W=\cup_{n\in\mN}g_nL_n\mdot\varphi(x)$. Note that
the sets $g_nL_n\mdot\varphi(x)$ are closed in~$W$ with respect to the relative
topology. By part~\textit{1.}\ of Lemma~\ref{Glimm_lem:Baire_spaces} we know
that $W$ is a Baire space. Hence there exists $n\in\mN$ such that
$g_nL_n\mdot\varphi(x)$ has a non-empty interior in~$W$ and consequently a
non-empty interior in~$Y$. As in the proof of Theorem~\ref{Glimm_thm:open_mapping_I}
we now conclude that $\varphi(U)$ is a $\varphi(x)$-\nb\ which proves
$\varphi$ to be open.\\\\
Finally we note that $G_x\subset G_{\varphi(x)}$ so that the natural projection
$G\to G/G_{\varphi(x)}$ factors to a continuous and open map
$\nu_x:G/G_x\onto G/G_{\varphi(x)}$. The commutative diagram
\begin{equation*}
\xymatrix{
G/G_x\ar[d]_{\omega_x}\ar[r]^{\nu_x}& G/G_{\varphi(x)}
\ar[d]^{\omega_{\varphi(x)}}\\
X\ar[r]^{\varphi} & Y}
\end{equation*}
\noindent shows us that $Y$ is a homogeneous space: Since $\nu_x$ is continuous,
$\varphi$ is an open continuous map and $\omega_x$ is a homeomorphism, it
follows that $\omega_{\varphi(x)}$ is a homeomorphism, too.
\end{proof}

\noindent If $G$ itself is hereditary Lindel\"of, then every homogeneous
$G$-space $X$ has this property. If $Y$ is almost Hausdorff and locally
quasi-compact, then $Y$ satisfies the assumption of the preceding theorem
by Lemma~\ref{Glimm_lem:almost_Hausdorff_locally_quasi_compact_implies_Baire}.

\subsection*{Glimm's Theorem}
In this section we will give a complete proof of Glimm's theorem which
contains necessary and sufficient condition for $G\setminus X$ to be almost
Hausdorff. The definition of almost Hausdorff spaces is due to to
J.\ Glimm, see~\cite{Glimm1}. A more distinct exposition can be
found in Section~7 of~\cite{Rief2} where M.\ Rieffel proves the
following neat result.

\begin{prop}[Marc A. Rieffel, 1979]
\label{Glimm_prop:locally_compact_implies_locally_closed}\hfill\\
Let $G$ be a topological group and $X$ a $G$-space. Let $x\in X$ be a point
such that the closure of $G\mdot x$ in $X$ contains a non-empty open Hausdorff
subset. If $G\mdot x$ is locally compact in the relative topology of $X$, then
$G\mdot x$ is locally closed in $X$.
\end{prop}
\begin{proof}
We can assume that $G\mdot x$ is dense in $X$. Let $y\in G\mdot x$ be arbitrary.
By assumption there exists a non-empty open Hausdorff subset $U$ of $X$. Clearly
$U\cap G\mdot x\neq\emptyset$. By translation of $U$ we can achieve
$y\in U$. Since $G\mdot x$ is locally compact, there is a compact $y$-\nb\ $V$
of $G\mdot x$ such that $V\subset U\cap G\mdot x$. Let $W$ be an open subset
of $U$ such that $W\cap G\mdot x$ is equal to the interior of $V$. We shall show that
$W\cap\overline{G\mdot x}\subset G\mdot x$: Let $z_\lambda$ be a net in $G\mdot x$
which converges to $z\in W$. Clearly $z_\lambda\in W\cap G\mdot x\subset V$ for
$\lambda\ge\lambda_0$. Since $V$ is compact, $z_\lambda$ has a subnet converging
to $z_0\in V$. For limits are unique in the Hausdorff subset $U$, it follows
$z=z_0\in G\mdot x$. This proves $G\mdot x$ to be locally closed.
\end{proof}

\noindent The open mapping theorem and the preceding proposition allow
us to characterize the well-behaved orbits of a totally Baire $G$-space under a
mild additional assumption of Hausdorffness. The next proposition is a slight
generalization of results of Glimm and Rieffel.

\begin{prop}
\label{Glimm_prop:characterization_of_smooth_orbits}
Let $G$ be a hereditary Lindel\"of locally compact group and $X$ a totally Baire 
$G$-space. Let $x\in X$ be a point such that the closure of the orbit $G\mdot x$
contains a non-empty open Hausdorff subset. Then there are equivalent:
\begin{enumerate}
\item The orbit $G\mdot x$ is a homogeneous space.
\item $G\mdot x$ is locally compact in the relative topology of $X$.
\item $G\mdot x$ is locally closed in $X$.
\item $G\mdot x$ is a $G_\delta$-subset of its closure.
\item $G\mdot x$ is a Baire space in the relative topology of $X$.
\end{enumerate}
\end{prop}
\begin{proof}
If the $G\mdot x$ is a homogeneous space, then $G\mdot x$ is locally compact
in the relative topology by Lemma~\ref{Glimm_lem:GmodH_is_Hausdorff}. This proves
$\textit{1.}\Rightarrow\textit{2.}$ Let $G\mdot x$ be locally compact. We assumed
that the closure of $G\mdot x$ contains a non-empty Hausdorff subset. Thus
Lemma~\ref{Glimm_prop:locally_compact_implies_locally_closed} implies that
$G\mdot x$ is locally closed. Hence $\textit{2.}\Rightarrow\textit{3.}$
Let $G\mdot x $ be locally closed. This means that $G\mdot x$ is open
in its closure. In particular, $G\mdot x$ is a $G_\delta$-subset of
its closure which shows $\textit{3.}\Rightarrow\textit{4.}$ Let $G\mdot x$
be a $G_\delta$-subset of $\overline{G\mdot x}$. Since $G\mdot x$ is dense
and $\overline{G\mdot x}$ is Baire as a closed subset of the totally Baire
space $X$, it follows from part~\textit{4.}\ of Lemma~\ref{Glimm_lem:Baire_spaces}
that $G\mdot x$ is Baire. This proves $\textit{4.}\Rightarrow\textit{5.}$
Let $G\mdot x$ be Baire. Since $G$ is hereditary Lindel\"of,
Theorem~\ref{Glimm_thm:open_mapping_III} implies that $G\mdot x$ is a
homogeneous space.
\end{proof}

\noindent Note that $\textit{5.}\Rightarrow\textit{1.}$ is the only implication
which requires $G$ to be hereditary Lindel\"of and locally compact.\\\\
In Proposition~\ref{Glimm_prop:characterization_of_smooth_orbits}
we encounter two different types of properties: The conditions
$G\mdot x$ is locally closed / a $G_\delta$-subset of its closure
concern the way in which $G\mdot x$ is situated in the space $X$
and hence the topological relation between different orbits. On the
other hand the conditions $G\mdot x$ is a Baire space / a homogeneous
space / locally compact are properties of a single orbit.\\\\
The validity of $\textit{1.}\Leftrightarrow\textit{3.}$ for actions of
$\sigma$-compact groups on locally compact Hausdorff spaces has been noted
on~p.~183 of~\cite{Foll}.

\begin{prop}
 Let $G$ be a $\sigma$-compact locally compact group and $X$ a locally compact
 $G$-space. Let $x\in X$ be arbitrary. Then the orbit $G\mdot x$ is a homogeneous
 space if and only if $G\mdot x$ is locally closed.
\end{prop}
\begin{proof}
 Since $X$ is Hausdorff, the closure of $G\mdot x$ contains a non-empty open
 Hausdorff subset. Consequently the implications $\textit{1.}\Rightarrow \textit{2.}
 \Rightarrow\textit{3.}\Rightarrow\textit{4.}\Rightarrow\textit{5.}$ can be proved
 as in the proof of Proposition~\ref{Glimm_prop:characterization_of_smooth_orbits}.
 Furthermore we can apply Theorem~\ref{Glimm_thm:open_mapping_I} to prove
 $\textit{5.}\Rightarrow\textit{1.}$ because $G$ is $\sigma$-compact and
 $X$ is locally compact (and in particular Hausdorff). 
\end{proof}

\noindent The next proposition contains the essential step in the
proof of Glimm's theorem. For the convenience of the reader we
reproduce Glimm's beautiful argument which can be found
in~\cite{Glimm1}.

\begin{prop}[Local transitivity of $e$-\nb s]
\label{Glimm_prop:slice_like_neighborhood}\hfill\\
Let $X$ be a $G$-space of a locally compact group such that all
$G$-orbits in $X$ are homogeneous spaces. Assume that $X$ is
second-countable, locally quasi-compact, and almost Hausdorff.
Then
\begin{enumerate}
\item For every non-empty, open subset $V$ of $X$ and every $e$-\nb\
$N$ of $G$ there exists a non-empty, open subset $U\subset V$ with
the following property: If $g\in G$ and if $U_0\subset U$ is
non-empty, open such that $g\mdot U_0\subset U$, then
$N\mdot U_0\cap g\mdot U_0\neq\emptyset$.
\item For every non-empty, open subset $V$ of $X$ and every $e$-\nb\
$N$ of $G$, there exists a non-empty, open subset $U\subset V$
such that $N\mdot x\cap U=G\mdot x \cap U$ for all $x\in U$.
\end{enumerate}
\end{prop}

\begin{proof}\hfill
\begin{enumerate}
\item Suppose that the assertion of \textit{1.} does not hold true so 
that there is a non-empty, open subset $V$ of $X$ and an $e$-\nb\
$N$ of $G$ with the following property: For every non-empty, open
subset $U\subset V$ there exist an element $g\in G$ and a non-empty,
open subset $U_0\subset U$ such that $g\mdot U_0\subset U$ and
$N\mdot U_0\cap g\mdot U_0=\emptyset$. Since $X$ is almost Hausdorff,
we can assume that $V$ is Hausdorff. For $X$ is second countable,
there is a basis $\{W_n:n\ge 1\}$ of the topology of $X$.\\\\
By induction we can choose elements $g_n\in G$ and compact subsets
$E_n\subset V$ with non-empty interior such that the following
conditions are satisfied:
\begin{enumerate}
\item $g_n\mdot E_n\subset E_{n-1}$ and $E_n\subset E_{n-1}$,
\item $N\mdot E_n\cap g\mdot E_n=\emptyset$,
\item $E_n\subset W_n$ or $E_n\cap W_n=\emptyset$.
\end{enumerate}
We shall explain the details: Assume that $E_1,\ldots,E_{n-1}$ are given
such that the conditions \textit{(a)-(c)} hold true. If $E_{n-1}\cap W_n
\neq\emptyset$, then $W_n$ has a non-empty intersection with the
interior $\overset{\circ}{E}_{n-1}$ of $E_{n-1}$ so that there
exist an element $g_n\in G$ and a non-empty, open subset
$U_n\subset\overset{\circ}{E}_{n-1}\cap W_n$ such that
$g_n\mdot U_n\subset\overset{\circ}{E}_{n-1}\cap W_n$ and
$N\mdot U_n\cap g_n\mdot U_n=\emptyset$ by the defining property
of $V$ and $N$. Since $X$ is locally quasi-compact, we can choose
a compact subset $E_n\subset U_n$ with non-empty interior which, by
definition, satisfies \textit{(a)-(c)}. On the other hand, if
$E_{n-1}\cap W_n=\emptyset$, we can find $g_n\in G$ and
$U_n\subset\overset{\circ}{E}_{n-1}$ non-empty, open such that
$N\mdot U_n\cap g\mdot U_n=\emptyset$. Now any compact subset
$E_n\subset U_n$ with non-empty interior satisfies our
requirements.\\\\
Since $E_1$ is compact and $E_n$ is a decreasing sequence
of non-empty closed subsets of $E_1$, there exists a point
$x\in\bigcap_{n=1}^\infty E_n$ by the finite intersection property.
Now condition \textit{(a)} and \textit{(c)} imply $g_n\mdot x\to x$,
and \textit{(b)} implies $g_n\mdot x\not\in N\mdot x$. Thus
$N\mdot x$ cannot be an $x$-\nb\ of $G\mdot x$, in contradiction
to the assumption that all $G$-orbits are homogeneous spaces.
This proves assertion \textit{1.}
\item Given $N$ and $V$, we fix a non-empty, open Hausdorff subset
$W_2\subset V$, which is possible because $X$ is almost Hausdorff.
Further we choose a non-empty, open subset $W_1\subset W_2$ and
a compact $e$-\nb\ $L\subset N$ of $G$ such that 
$L\mdot W_1\subset W_2$. Here we use the fact that $G$ is
locally compact. Further it follows form part \textit{1.} that
there exist a non-empty, open subset $U\subset W_1$ with the
following property: If $g\in G$ and $U_0\subset U$ is non-empty,
open such that $g\mdot U_0\subset U$, then it follows
$L\mdot U_0\cap g\mdot U_0\neq\emptyset$.\\\\
Let $x\in U$ be arbitrary. We must prove $G\mdot x\cap U\subset
L\mdot x \cap U$. Let $g\in G$ be arbitrary such that $g\mdot x\in U$.
If $\{U_n:n\ge 1\}$ is a basis of $x$-\nb s, then $g\mdot U_n\subset U$
for large $n$, and thus $L\mdot U_n\cap g\mdot U_n\neq\emptyset$
by definition of $U$. Hence there exist $h_n\in L$ and
$x_n,y_n\in U_n$ such that $g\mdot x_n=h_n\mdot y_n$. Since $L$
is compact, we can assume that $h_n\to h$ converges in $L$.
Since $W_2$ is Hausdorff and $x_n,y_n\to x$, we see that
\[g\mdot x=\lim_{n\to\infty} g\mdot x_n=\lim_{n\to\infty}
h_n\mdot y_n=h\mdot x\]
lies in $L\mdot x\cap U$, which completes the proof.
\end{enumerate}
\end{proof}

\noindent We shall give a further explanation of condition \textit{2}:
For any open subset $V$ of $X$ and any $e$-\nb\ $N$ of $G$ there
is an open subset $U$ of $V$ such that $N\mdot x\cap U=G\mdot x\cap U$
for every $x\in U$. In Glimm's own words: $N$ acts locally as transitive
as $G$ does. One should think of $U$ as a prolate open subset which is
transversal to the orbits passing through $V$ and whose width becomes
arbitrarily small as $N$ shrinks to $\{e\}$.\\\\
Now we present a reformulation of Glimm's theorem characterizing nice
behaviour of $G$-spaces, see Theorem~1 of~\cite{Glimm1}. In contrast
to the original proof we shall avoid measure theoretical arguments involving
the Borel structure of the orbit space $G\setminus X$ and quasi-invariant
ergodic measures on $X$.

\begin{thm}[James Glimm, 1961]
\label{Glimm_thm:characterization_of_nice_actions}\hfill\\
Let $G$ be a hereditary Lindel\"of  locally compact group and $X$ a
second countable, locally quasi-compact, almost Hausdorff $G$-space.
Then there are equivalent:
\begin{enumerate}
\item $G\setminus X$ is almost Hausdorff.
\item $G\setminus X$ is a $T_0$-space.
\item Every orbit is a homogeneous space.
\end{enumerate}
\end{thm}
\begin{proof}
The implication $\textit{1.}\Rightarrow\textit{2}$.\ follows from
part~\textit{1.}\ of Lemma~\ref{Glimm_lem:almost_Hausdorff_spaces}.
For $\textit{2.}\Rightarrow\textit{3.}$\ we assume that $G\setminus X$ is a
$T_0$-space. Let $\pi:X\onto G\setminus X$ denote the projection onto the
orbit space which is an open map. Since $X$ and hence $G\setminus X$ are
first countable, the point $G\mdot x$ of $G\setminus X$ has a countable
neighborhood basis $\dot{U}_n$. For $G\setminus X$ is $T_0$, we obtain
$\{G\mdot x\}=\bigcap_{n=1}^\infty\dot{U}_n\cap
\{G\mdot x\}\clos$. This implies that
\[G\mdot x=\bigcap_{n=1}^\infty\;\pi^{-1}(\dot{U}_n)
\cap\overline{G\mdot x}\]
is a $G_\delta$-subset of its closure. Furthermore $X$ is locally
quasi-compact almost Hausdorff and hence totally Baire by
Proposition~\ref{Glimm_lem:almost_Hausdorff_locally_quasi_compact_implies_Baire}.
Since $G$ is hereditary Lindel\"of,
Proposition~\ref{Glimm_prop:characterization_of_smooth_orbits} implies
that $G\mdot x$ is a homogeneous space.\\\\
Next we verify \textit{3.}$\Rightarrow$\textit{1.} We must
prove that any non-empty closed subset $\dot{A}$ of the orbit
space $G\setminus X$ contains a non-empty relatively open Hausdorff
subset $\dot{W}$. Since $A=\pi^{-1}(\dot{A})$ is $G$-invariant
and closed, it is clear that $A$ is a second countable, locally
quasi-compact, almost Hausdorff $G$-space satisfying \textit{3}.
First we choose a non-empty relatively open Hausdorff subset
$W_2$ of $A$. Fix a compact $e$-\nb\ $N$ of $G$ and a
non-empty relatively open subset $W_1$ of $W_2$ such that
$N\mdot W_1\subset W_2$. By applying Proposition~
\ref{Glimm_prop:slice_like_neighborhood} to the $G$-space
$A$ we find a non-empty relatively open subset $U$ of $W_1$
such that $N\mdot x\cap U=G\mdot x\cap U$ for all $x\in U$.
Finally we define $W=G\mdot U$ which is relatively open in $A$.
Suppose that $\dot{W}=\pi(W)$ is not Hausdorff so that there
exist points $x,y\in W$ such that $G\mdot x\neq G\mdot y$
cannot be separated by disjoint $G$-invariant open \nb s.
For $X$ is first countable, we find open sets
$X_n, Y_n\subset U$ such that $\{x\}=\bigcap_{n=1}^\infty X_n$
and $\{y\}=\bigcap_{n=1}^\infty Y_n$. By assumption we have
$G\mdot X_n\cap G\mdot Y_n\neq\emptyset$ so that there exist
elements $x_n\in X_n$, $y_n\in Y_n$, $g_n\in G$ such that
$x_n=g_n\mdot y_n$. Using $N\mdot y_n\cap U=G\mdot y_n\cap U$
we can even find $h_n\in N$ such that $x_n=h_n\mdot y_n$.
Since $N$ is compact, we can assume that $h_n\to h$
converges in $N$. This implies
\[x=\lim_{n\to\infty}x_n=\lim_{n\to\infty}h_n\mdot y_n
=h\mdot y\]
and hence $G\mdot x=G\mdot y$, a contradiction. Thus
$\dot{W}=\pi(W)$ must be Hausdorff. This proves $G\setminus X$ to be
almost Hausdorff. The proof of Glimm's theorem is complete.
\end{proof}

\noindent Note that for $\textit{1.}\Rightarrow\textit{2.}\Rightarrow\textit{3.}$
we do not need that $X$ is second countable.\\\\
In retrospect, Glimm's merit is the ingenious proof of
Proposition~\ref{Glimm_prop:slice_like_neighborhood} which implies that
$G\setminus X$ is almost Hausdorff provided that all orbits are homogeneous spaces.\\\\
Next we shall be concerned with the question whether a given $G$-space $X$
is quasi-regular or regular following the ideas of P.\ Green in section~5
of~\cite{Green1}. To this end we introduce the notion of $G$-irreducible subsets
and generic orbits.

\begin{defn}
 Let $G$ be a topological group and $X$ a $G$-space. A $G$-invariant closed
 subset $W$ of $X$ is called $G$-irreducible if it cannot be written as a
 union of two proper $G$-invariant closed subsets, i.e., if $W=A\cup B$ with
 $A,B\subset X$ $G$-invariant and closed implies $A=X$ or $B=X$. The subset~$W$
 has a generic orbit if there exists  a (not necessarily unique) $G$-orbit
 $G\mdot x$ such that $W=\overline{G\mdot x}$.
\end{defn}

\noindent By definition $W$ is $G$-irreducible if and only if $G\setminus W$ is
an irreducible subset of $G\setminus X$ and $W$ has a generic orbit if and only
if $G\setminus W$ has a generic point. These notions are borrowed from algebraic
geometry.\\\\
The following definition is from~\cite{Green1} and can also be found on p.~186
of~\cite{Will} where the role of quasi-orbits is emphasized.

\begin{defn}\label{Glimm_defn:regular_G_space}
 A $G$-space $X$ is quasi-regular if every $G$-irreducible subset has
 at least one generic orbit. A $G$-space $X$ is called regular if the 
 following three conditions are satisfied:
 \begin{itemize}
  \item $X$ is quasi-regular.
  \item Every $G$-orbit of $X$ is locally closed.
  \item Every $G$-orbit is a homogeneous space.
 \end{itemize}
\end{defn}

\noindent We emphasize that the first two conditions in the definition 
of a regular $G$-space imply that every $G$-irreducible subset contains 
a unique generic orbit. The last two conditions of 
Definition~\ref{Glimm_defn:regular_G_space} imply that all five assertions 
stated in Proposition~\ref{Glimm_prop:characterization_of_smooth_orbits} hold 
true.\\\\
If $G$ is hereditary Lindel\"of and $X$ is totally Baire, then it suffices to assume
that every $G$-orbit contains a non-empty open Hausdorff subset instead of the
stronger condition that every orbit is a homogeneous space. If, in addition, $X$
is almost Hausdorff, then for $X$ to be regular it suffices to assume that $X$
is quasi-regular and every $G$-orbit is locally closed.

\begin{prop}\label{Glimm_prop:almost_Hausdorff_implies_quasi_regular}
Let $X$ be a $G$-space. If the orbit space $G\setminus X$ is almost Hausdorff,
then $X$ is quasi-regular and every orbit is locally closed. 
\end{prop}
\begin{proof}
Let $W$ be a $G$-irreducible subset of~$X$. Since $G\setminus W$ is almost
Hausdorff, there exists a non-empty $G$-invariant dense open subset~$U$ of~$W$
such that $G\setminus U$ is Hausdorff. Suppose that $U$ contains two distinct
orbits $G\mdot x$ and $G\mdot y$. Then there exist two disjoint $G$-invariant
open subsets $U_1$ and $U_2$ of $U$ such that $x\in U_1$ and $y\in U_2$. Now it
follows $W=(W\setminus U_1)\cup(W\setminus U_2)$ in contradiction to
the $G$-irreducibility of $W$. Thus $U=G\mdot x$ for some $x\in X$.
Clearly $G\mdot x$ is a generic orbit for $W$ because $U$ is dense.\\
Let $x\in X$ be arbitrary. Then $W:=\overline{G\mdot x}$ is a $G$-irreducible 
subset of $X$. The preceding considerations show that $U=G\mdot x$ is open 
in~$W=\overline{G\mdot x}$. Thus $G\mdot x$ is locally closed. 
Alternatively, this can be seen as a consequence of part~\textit{1.}\ of 
Lemma~\ref{Glimm_lem:almost_Hausdorff_spaces}.
\end{proof}

\begin{prop}\label{Glimm_prop:almost_Hausdorff_implies_regular}
Let $G$ be a hereditary Lindel\"of locally compact group and $X$ a locally
quasi-compact almost Hausdorff $G$-space such that $G\setminus X$ is almost
Hausdorff. Then $X$ is regular.
\end{prop}
\begin{proof}
By Proposition~\ref{Glimm_prop:almost_Hausdorff_implies_quasi_regular} we know
that $X$ is quasi-regular and that every orbit is locally closed. Let $x\in X$
be arbitrary. Since $G$ is hereditary Lindel\"of and $X$ is totally Baire by
Lemma~\ref{Glimm_lem:almost_Hausdorff_locally_quasi_compact_implies_Baire},
it follows from Proposition~\ref{Glimm_prop:characterization_of_smooth_orbits}
that $G\mdot x$ is a homogeneous space.
\end{proof}

\subsection*{Actions of compact groups}

In this section we are interested in $G$-spaces of compact groups. We
begin with a technical lemma.

\begin{lem}\label{Glimm_lem:Ginvariant_Hausdorff_subsets}
Let $G$ be a compact group and $X$ a $G$-space which contains a dense
open Hausdorff subset $U$. Then there exists a non-empty $G$-invariant
open Hausdorff subset~$W$ of $X$. 
\end{lem}
\begin{proof}
Let $x_0$ be in $U$. Since $G\mdot x_0$ is quasi-compact, there
are $g_1,\ldots g_n\in G$ such that $W_0=\bigcup_{k=1}^n g_k\mdot U$
contains $G\mdot x_0$. Further there exists an open $x_0$-\nb\ $V_0$ such
that $G\mdot V_0\subset W_0$ because $W_0$ is open and $G$ is compact.
As $\bigcap_{k=1}^n g_k\mdot U$ is dense in $X$, it follows that
$V=V_0\cap\left(\bigcap_{k=1}^n g_k\mdot U\right)$ is non-empty and
open. Define $W:=G\mdot V$. Clearly $W$ is non-empty, $G$-invariant
and open. We prove that $W$ is Hausdorff: Let $y_1\neq y_2$ be
in $W$. Then $y_1=g\mdot x_1$ for some $x_1\in V$ and $g\in G$. Put
$x_2:=g^{-1}\mdot y_2$. Then $x_2\in g_k\mdot U$ for some $1\le k\le n$.
Since $g_k\mdot U$ is Hausdorff and $V\subset g_k\mdot U$, there are
disjoint open \nb s $X_1$ and $X_2$ of $x_1$ and $x_2$ respectively.
Finally $g\mdot X_1$ and $g\mdot X_2$ are disjoint open \nb s of
$y_1$ and $y_2$ which proves $W$ to be Hausdorff.
\end{proof}

\noindent Now we are able to prove the main result of this section.

\begin{thm}\label{Glimm_thm:compact_groups_act_almost_properly}
Let $G$ be a compact group and $X$ an almost Hausdorff $G$-space. Then
the orbit space $G\setminus X$ is almost Hausdorff, too. Moreover, $X$ is
a regular $G$-space. 
\end{thm}
\begin{proof}
First we prove that $G\setminus X$ is almost Hausdorff: Let $A$ be a
$G$-invariant closed subset of $X$. Clearly $A$ contains a
relatively open dense Hausdorff subset $U$. By the preceding
lemma it follows that $A$ contains a non-empty, $G$-invariant,
relatively open Hausdorff subset $W$. We claim that $G\setminus W$
is Hausdorff: If not, then there are points $x$ and $y$ of $W$
such that $G\mdot x\neq G\mdot y$ cannot be separated by
disjoint $G$-invariant \nb s. We choose nets $X_\lambda$ and
$Y_\lambda$ of \nb s of $x$ and $y$ respectively such that
$\{x\}=\bigcap_{\lambda\in\Lambda} X_\lambda$ and
$\{y\}=\bigcap_{\lambda\in\Lambda} Y_\lambda$. For every
$\lambda\in\Lambda$ we find $x_\lambda\in X_\lambda$,
$y_\lambda\in Y_\lambda$, and $g_\lambda\in G$ such that
$x_\lambda=g_\lambda\mdot y_\lambda$ because
$G\mdot X_\lambda\cap G\mdot Y_\lambda\neq\emptyset$.
Since $G$ is compact, there is a subnet $g_{\lambda_\nu}$
which converges to $g\in G$. Using that limits are unique in
the Hausdorff subset $W$ we conclude 
\[x=\lim_{\nu}x_{\lambda_\nu}=\lim_\nu g_{\lambda_\nu}\mdot
y_{\lambda_\nu}=g\mdot y\]
and hence $G\mdot x=G\mdot y$, a contradiction. Thus $G\setminus X$
is almost Hausdorff. In particular, $X$ is quasi-regular by
Proposition~\ref{Glimm_prop:almost_Hausdorff_implies_quasi_regular}.\\\\
Finally we verify that all orbits are homogeneous spaces. Let
$x\in X$ be arbitrary. Since $G\setminus X$ is almost Hausdorff, there
exists a non-empty, $G$-invariant, relatively open Hausdorff
subset $W$ of $\overline{G\mdot x}$ such that $G\setminus W$ is
Hausdorff. This implies that $G\mdot x=W$ is Hausdorff and locally
closed. Moreover it follows from Theorem~\ref{Glimm_thm:open_mapping_I}
that $G\mdot x$ is a homogeneous space. Altogether, this shows that
$X$ is a regular $G$-space.
\end{proof}

\begin{defn}\label{Glimm_defn:almost_proper_G_space}
 A $G$-space $X$ is called almost proper if every closed $G$-invariant
 subset of $X$ contains a non-empty open $G$-invariant Hausdorff subset $W$
 such that $G$ acts properly on $W$ in the sense that the
 map $G\times W\to W\times W$, $(g,x)\mapsto(g\mdot x,x)$, is proper.
\end{defn}

\noindent By Zorn's Lemma $W$ can be chosen to be dense in $X$.\\\\
If $X$ is an almost proper $G$-space, then, in particular, $X$
is almost Hausdorff.

\begin{lem}\label{Glimm_lem:almost_proper_G_spaces_are_regular}
 Let $X$ be an almost proper $G$-space. Then $X$ is regular.
\end{lem}
\begin{proof}
 Let $x\in X$ be arbitrary. By definition there exists a non-empty open
 $G$-invariant Hausdorff subset $W$ of $\overline{G\mdot x}$ such that
 $G$ acts properly on $W$. The basic results about proper actions imply
 that the orbit space $G\setminus W$ is Hausdorff. As in the proof of
 Theorem~\ref{Glimm_thm:compact_groups_act_almost_properly} it follows that 
 $G\mdot x=W$ is Hausdorff and locally closed. Moreover, we find that 
 $G\mdot x$ is a homogeneous space, either as a basic result about proper 
 actions or as a consequence of Theorem~\ref{Glimm_thm:open_mapping_I}.
\end{proof}

\noindent Lemma~\ref{Glimm_lem:Ginvariant_Hausdorff_subsets} implies
that that almost Hausdorff $G$-spaces of compact groups are almost proper.

\subsection*{Applications in harmonic analysis}

Important examples of topological spaces which are not Hausdorff are
primitive ideal spaces of $C^\ast$-algebras.\\\\
Let $\mcA$ be a $C^\ast$-algebra. Let $\Prim\mcA$ denote the space 
of all primitive ideals of $\mcA$ endowed with the hull-kernel topology.
Let $\widehat{\mcA}$ denote the set of all equivalence classes of topologically
irreducible $\ast$-representations of $\mcA$ endowed with the initial topology
with respect to the canonical map $\widehat{\mcA}\to\Prim\mcA$. Obviously
$\Prim\mcA$ is a $T_0$-space. It is well-known that $\widehat{\mcA}$ and 
hence $\Prim\mcA$ are locally quasi-compact. The spectrum $\widehat{\mcA}$
of $\mcA$ is in general not $T_0$. Furthermore the locally closed subsets
of $\Prim\mcA$ can be characterized easily: If $A\subset\Prim\mcA$ is
locally closed, then there exist closed ideals $I$ and $J$ of $\mcA$ such
that $A=h(J)\cap (\Prim\mcA\setminus h(I))$. Thus $A$ is homeomorphic to
the primitive ideal space of the subquotient $(I+J)/J$ of~$\mcA$.\\\\
The following assertion is proved in Corollaire~3.4.13 of~\cite{Dix3}.

\begin{lem}
 The space $\Prim\mcA$ is totally Baire.
\end{lem}
\begin{proof}
Let $A$ be a closed subset of $\Prim\mcA$. Let $I$ be an ideal of $\mcA$
such that $A=h(I)$. Define $\mcA_0:=\mcA/I$. We prove that $\Prim\mcA_0$ is
Baire. Let $V_n$ be a sequence of dense open subset of $\Prim\mcA_0$. 
Let $P(\mcA_0)$ denote set of all pure states of $\mcA_0$ endowed with
the relative topology of the $\sigma(\mcA_0',\mcA_0)$-topology of $\mcA_0'$.
Let $U_n$ be the preimage of $V_n$ in~$P(\mcA_0)$. Since the canonical maps
$P(\mcA_0)\to\widehat{\mcA}_0\to\Prim\mcA_0$ are continuous and open, it
follows that $U_n$ is dense and open. By Choquet's Theorem we know that
$P(\mcA_0)$ is a Baire space, see Appendix~(B14) of~\cite{Dix3}. This
implies that $\cap_{n=1}^\infty U_n$ is dense in $P(\mcA_0)$. Hence
$\cap_{n=1}^\infty V_n$ is dense in $\Prim\mcA_0$. This proves that
$A=h(I)\cong\Prim\mcA_0$ is a Baire space.
\end{proof}

\noindent We recall the fundamental results about primitive ideal spaces.
Let $\mcA$ be a $C^\ast$-algebra. We say that $\mcA$ is liminal if, for every 
irreducible $\ast$-representation $\pi$ of~$\mcA$ in a Hilbert space $\frakH$,
the image $\pi(\mcA)$ is contained in the $\Cst$-algebra $K(\frakH)$ of all 
compact operators. Furthermore we say that $\mcA$ is postliminal if every
non-zero quotient of $\mcA$ contains a non-zero liminal ideal. It is known 
that if $\mcA$ is postliminal, then $\widehat{\mcA}\cong\Prim\mcA$ is almost 
Hausdorff, see e.g.\ Theoreme 4.4.5 of~\cite{Dix3}. If $\mcA$ is liminal, 
then $\widehat{\mcA}\cong\Prim\mcA$ is $T_1$, compare Corollaire~4.1.10
of~\cite{Dix3}. Furthermore, if $\mcA$ admits a continuous trace, then 
$\mcA$ is liminal and $\widehat{\mcA}\cong\Prim\mcA$ is Hausdorff, see 
Proposition~4.5.3 of~\cite{Dix3}.\\\\
Glimm's intention was to study group actions on primitive ideal spaces
and  to find sufficient conditions for the orbit space $\Prim\mcA\,/\,G$ 
to be almost Hausdorff.\\\\
Let $(G,\mcA)$ be a covariance system, i.e., $\mcA$ is a $C^\ast$-algebra
and $G$ a locally compact group acting strongly continuously on $\mcA$ as a
group of (isometric) automorphisms. It is known that $\Prim\mcA$ becomes a 
$G$-space in a natural way. The $\Cst$-covariance algebra $\Cst(G,\mcA)$ 
associated to $(G,\mcA)$ is defined as follows: Let $\mcC_0(G,\mcA)$ denote
the vector space of all $\mcA$-valued continuous functions on $G$ with
compact support. We endow $\mcC_0(G,\mcA)$ with the multiplication
\[(f\ast g)(x)=\int_G f(xy)^{y^{-1}}g(y^{-1})\;dy\]
and the involution
\[f^\ast(x)=\Delta_G(x^{-1})\;(f(x^{-1})^\ast)^{x}\;.\]
Let $L^1(G,\mcA)$ denote the Banach $\ast$-algebra obtained as the 
completion of $\mcC_0(G,\mcA)$ with respect to the norm
\[|\,f\,|_1=\int_G|f(x)|\;dx.\]
Let $|\,f\,|_\ast=\sup\left\{\;|\,\pi(f)\,|\;:\;\pi\text{ is a bounded}
\ast\text{-representation of }L^1(G,\mcA)\;\right\}$ denote the
$C^\ast$-semi-norm of~$L^1(G,\mcA)$ and $R:=\{f\in L^1(G,\mcA):|\,f\,|_\ast=0\}$
the $\ast$-radical. Then $\Cst(G,\mcA)$ is defined as the completion
of $L^1(G,\mcA)/R$ with respect to the norm induced by $|\;\;|_\ast$.\\\\
In the following we recall the properties of the extension and restriction 
process for ideals of covariance algebras. We refer to~\cite{Green1} for 
a more extensive treatment including the induction of ideals using
imprimitivity bimodules and the concept of strong Morita equivalence.\\\\
Every non-degenerate $\ast$-representation $\pi$ of $C^\ast(G,\mcA)$ has 
the form
\[\pi(f)\xi=\int_G\pi_1(x)\pi_2(f(x))\xi\,dx\]
for a suitable covariance pair $(\pi_1,\pi_2)$ of $(G,\mcA)$, i.e., a
strongly continuous unitary representation $\pi_1$ of $G$ and a non-dengenerate
$\ast$-representation $\pi_2$ of $\mcA$ on the same Hilbert space $\frakH$
satisfying $\pi_2(a^x)=\pi_1(x)^\ast\pi_2(a)\pi(x)$. Using this fact one can 
prove that $\Cst(G,\mcA)$ becomes an $\mcA$-bimodule by means of the actions
\[(a\ast f)(x)=a^xf(x)\quad\text{and}\quad (f\ast a)(x)=f(x)a\;.\]
In particular, it holds $(ab)\ast f=a\ast(b\ast f)$,
$a\ast(f\ast g)=(a\ast f)\ast g$, $|a\ast f|_\ast\le |a|\;|f|_\ast$,
$(a\ast f)^\ast=f^\ast\ast a^\ast$, $(f\ast a)^\ast=a^\ast\ast f^\ast$,
$(a\ast f)\ast b=a\ast(f\ast b)$ and $(f\ast a)\ast g=f\ast(a\ast g)$
for all $f,g\in\Cst(G,\mcA)$ and $a,b\in\mcA$.\\\\
Furthermore we observe that
\[f^z(x)=\Delta_G(z^{-1})\,f(x^{z^{-1}})^z\]
defines a continuous right action of $G$ on $\Cst(G,\mcA)$. Note that $G$
acts a group of isometric isomorphisms. In particular, it holds
$(f\ast g)^z=f^z\ast g^z$ and $(f^\ast)^z=(f^z)^\ast$ for all 
$f,g\in\Cst(G,\mcA)$ and $z\in G$. Moreover, $(a\ast f)^z=a^z\ast f^z$
for $a\in\mcA$.\\\\
Let $I$ be a two-sided closed ideal of $\Cst(G,\mcA)$. In the sequel all 
ideals are assumed to be two-sided and closed. We define the restriction 
of $I$ to $\mcA$ by
\[\res(I):=\{a\in\mcA:a\ast\Cst(G,\mcA)\subset I\}\,.\]
Clearly $\res(I)$ is a closed and two-sided ideal of $\mcA$.
\begin{lem}
 Let $I$ be an ideal of $\Cst(G,\mcA)$. Then the restriction $\res(I)$
 is a $G$-invariant ideal of $\mcA$.
\end{lem}
\begin{proof}
First we prove that $I$ is $G$-invariant. Let $f\in I$. Let $(g_\lambda)$ 
be an approximate identity of $\Cst(G,\mcA)$. Consider the operators
$\lambda(z)f\,(x)=f(z^{-1}x)$ and $\rho(z)f\,(x)=\Delta_G(z)\,f(xz)^{z^{-1}}$.
Since $\rho(z^{-1})$ is a bounded right multiplier, it follows
\[\rho(z^{-1})f=\lim_\lambda\rho(z^{-1})(f\ast g_\lambda)
 =\lim_\lambda f\ast(\rho(z^{-1})g_\lambda)\in I\,.\]
Using that $\lambda(z^{-1})$ is a bounded left multiplier we obtain
$f^z=\lambda(z^{-1})\rho(z^{-1})f\in I$.\\\\
Now let $a\in\mcA$ be arbitrary. Then it follows
$a^z\ast\Cst(G,\mcA)=\left(a\ast\Cst(G,\mcA)\right)^z\subset I^z=I$ and 
hence $a^z\in J$. This proves $J$ to be $G$-invariant.
\end{proof}

\noindent Let $J$ be an ideal of $\mcA$. The extension of $J$ from $\mcA$ 
up to $\Cst(G,\mcA)$ is defined as the closure of the linear span of the 
set $\{f\ast a\ast g:a\in J\text{ and }f,g\in\Cst(G,\mcA)\}$, i.e.,
\[\ext(J):=\langle\,\Cst(G,\mcA)\ast J\ast\Cst(G,\mcA)\,\rangle\clos\;.\]
Clearly $J\subset\res(\ext(J))$ for all ideals $J$ of $\mcA$, and 
$\ext(\res(I))\subset I$ for all ideals $I$ of~$\Cst(G,\mcA)$.
By definition $\ext(J)$ is the smallest ideal $I$ of $\Cst(G,\mcA)$ which
satisfies $J\subset\res(I)$. 

\begin{lem}\label{Glimm_lem:induced_ideal}
Let $J$ be a $G$-invariant ideal of $\mcA$.
\begin{enumerate}
 \item The ideal $\ext(J)$ coincides with the closure of $\mcC_0(G,J)$
 in~$\Cst(G,\mcA)$.
 \item The set $\Cst(G,\mcA)\ast J$ is contained in the closure 
 of $J\ast\Cst(G,\mcA)$, and $J\ast\Cst(G,\mcA)$ in the closure
 of $\Cst(G,\mcA)\ast J$.
\end{enumerate}
\end{lem}
\begin{proof}
Since $\mcC_0(G,\mcA)\ast J\ast\mcC_0(G,\mcA)\subset\mcC_0(G,J)$, we see 
that $\ext(J)\subset\mcC_0(G,J)\clos$. Let $\pi\widehat{=}(\pi_1,\pi_2)$ 
be an arbitrary non-degenerate $\ast$-representation such that 
$\ext(J)\subset\ker\pi$. Then we obtain 
$J\subset\res(\ext(J))\subset\ker\pi_2$. Using 
$\pi(f)\xi=\int_G\pi_1(x)\pi_2(f(x))\xi\,dx$ we conclude
$\mcC_0(G,J)\subset\ker\pi$. Since $\ext(J)$ is an ideal 
of $\Cst(G,\mcA)$ and $\pi$ is arbitrary with $\ext(J)\subset\ker\pi$, 
it follows $\mcC_0(G,J)\subset\ext(J)$. This proves the first assertion.\\\\
Now let $u_\lambda$ be an approximating identity of $J$, i.e., a net
in $J$ such that $u_\lambda^\ast=u_\lambda$, $|u_\lambda|\le 1$,
and $|\;b-u_\lambda b\;|\to 0$ for every $b\in J$. Let
$f\in\mcC_0(G,\mcA)$ and $a\in J$ be arbitrary. Since $J$ is
$G$-invariant, it follows 
\[|\;f(x)a-u_\lambda^x\,f(x)a\;|=|\;(f(x)a)^{x^{-1}}-u_\lambda\,
(f(x)a)^{x^{-1}}\;|\to 0\]
for every $x\in G$. Thus
\[|\,f\ast a-u_\lambda\ast(f\ast a)\,|_\ast\le
|\,f\ast a-u_\lambda\ast(f\ast a)\,|_1=
\int_G|\,f(x)a-u_\lambda^xf(x)a\,|\,dx\to 0\,.\]
Since $u_\lambda\ast(f\ast a)$ is in $J\ast\Cst(G,\mcA)$, this proves
the first inclusion. The second one can be verified similarly.
\end{proof}

\noindent Moreover, it holds $\Cst(G,\mcA)/\ext(J)\cong\Cst(G,\mcA/J)$.\\\\
An ideal $I$ of $\Cst(G,\mcA)$ is said to be prime if
$I_1\ast I_2\subset I$ ideals $I_1$ and $I_2$ of $\Cst(G,\mcA)$ 
implies $I_1\subset I$ or $I_2\subset I$.

\begin{lem}\label{Glimm_lem:factor_representations}
If $\pi$ is a factor representation of a $\Cst$-algebra $\mcB$,
then $\ker\pi$ is prime.
\end{lem}
\begin{proof}
Let $I_1$ and $I_2$ be ideals of $\mcB$ such that
$I_1I_2\subset\ker\pi$ and $I_2\not\subset\ker\pi$. Then
$\pi(I_2)\frakH$ is a non-zero closed subspace of the
representation space $\frakH$ of $\pi$ which is $\pi(\mcB)\,$-
and $\pi(\mcB)\,'\,$-invariant. If $P$ denotes the orthogonal
projection onto the subspace $\pi(I_2)\frakH$, then $P$ is
non-zero and in $\pi(\mcB)'\cap\pi(\mcB)''=\mC\mdot\Id$
which is trivial because $\pi$ is a factor representation.
Thus $P=\Id$ which means $\pi(I_2)\frakH=\frakH$. Since
$\pi(I_1)\pi(I_2)=\pi(I_1I_2)=0$,
it follows $I_1\subset\ker\pi$. The proof is complete.
\end{proof}

\noindent A $G$-invariant ideal $J$ of $\mcA$ is called
$G$-prime if $J_1J_2\subset J$ for $G$-invariant ideals $J_1$ and
$J_2$ of $\mcA$ implies $J_1\subset J$ or $J_2\subset J$. Obviously
the hull $h(J)$ of a $G$-invariant ideal $J$ of $\mcA$ is a closed, 
$G$-invariant subset of $\Prim\mcA$, and the kernel $k(A)$ of a 
$G$-invariant subset $A$ of $\Prim\mcA$ is a $G$-invariant ideal 
of $\mcA$.\\\\
In order to prepare the proof of 
Theorem~\ref{Glimm_thm:restrictions_of_ideals_are_orbits} we note

\begin{prop}\label{Glimm_prop:G_prime_ideals}
Let $(G,\mcA)$ be a covariance system.
\begin{enumerate}
\item If $P$ is a prime ideal of $\Cst(G,\mcA)$,
then $J=\res(P)$ is a $G$-prime ideal of $\mcA$.
\item If $J$ is a $G$-prime ideal of $\mcA$, then $h(J)$ is
a $G$-irreducible subset of $\Prim\mcA$.
\item Suppose that $G\setminus\Prim\mcA$ is almost Hausdorff. Then
for every prime ideal~$P$ of~$\Cst(G,\mcA)$ there is a unique
orbit $G\mdot Q$ of $\Prim\mcA$ such that $\res(P)=k(G\mdot Q)$.
\end{enumerate}
\end{prop}
\begin{proof}\hfill
\begin{enumerate}\renewcommand{\labelenumi}{\textit{\arabic{enumi}}.}
\item Let $J_1,J_2$ be $G$-invariant ideals of $\mcA$ such that
$J_1J_2\subset J$. Now Lemma~\ref{Glimm_lem:induced_ideal} implies
$\ext(J_1)\ext(J_2)=\ext(J_1J_2)\subset\ext(J)\subset P$.
Since $P$ is prime, we conclude that $\ext(J_1)\subset P$ or
$\ext(J_2)\subset P$. Consequently $J_1\subset\res(\ext(J_1))\subset J$
or $J_2\subset\res(\ext(J_2))\subset J$. This proves $J$ to be $G$-prime.
\item Let $A_1,A_2$ be closed, $G$-invariant subsets of $h(J)$ such
that $A_1\cup A_2=h(J)$. Clearly $J_1=k(A_1)$ and $J_2=k(A_2)$ are
$G$-invariant ideals. Further $J_1J_2\subset J_1\cap J_2=J$. Since
$J$ is $G$-prime, we conclude $J_1=J$ or $J_2=J$, and hence
$A_1=h(J)$ or $A_2=h(J)$. Thus $h(J)$ is $G$-irreducible.
\item Let $P$ be a prime ideal of~$\Cst(G,\mcA)$. By part~\textit{1.}\ 
and \textit{2.}\ we know that $J:=\res(P)$ is a $G$-prime ideal 
of $\mcA$ and that $h(J)$ is a $G$-irreducible subset of $\Prim\mcA$. 
Since $G\setminus\Prim\mcA$ is almost Hausdorff, it follows from 
Proposition~\ref{Glimm_prop:almost_Hausdorff_implies_quasi_regular}
that $h(J)$ admits a unique generic orbit $G\mdot Q$ with $Q\in\Prim\mcA$.
This means $h(J)=\overline{G\mdot Q}$. Hence we obtain 
$J=k(h(J))=k(G\mdot Q)$. 
\end{enumerate}
\end{proof}

\noindent We note that the proof of part~\textit{3.}\ of the preceding proposition 
does not rely on a solution of Dixmier's problem whether every prime ideal of a 
$\Cst$-algebra is primitive.\\\\
The local closedness of all orbits is a feasible criterion for the 
almost Hausdorffness of the orbit space.

\begin{thm}\label{Glimm_thm:restrictions_of_ideals_are_orbits}
Let $(G,\mcA)$ be  a covariance system where $G$ is a hereditary Lindel\"of, 
locally compact group and $\mcA$ a separable $\Cst$-algebra such that
$\Prim\mcA$ is almost Hausdorff. Assume that all orbits in $\Prim\mcA$ 
are locally closed. Then $\Prim\mcA$ is a regular $G$-space and for every 
prime ideal $P$ of $\Cst(G,\mcA)$ there exists a unique orbit $G\mdot Q$ 
of~$\Prim\mcA$ such that $\res(P)=k(G\mdot Q)$.
\end{thm}
\begin{proof}
Since $\mcA$ is separable, it follows that $\Prim\mcA$ is a second 
countable, locally quasi-compact, almost Hausdorff $G$-space. By assumption
all orbits of $\Prim\mcA$ are locally closed. Since $G$ is hereditary 
Lindel\"of and $\Prim\mcA$ is almost Hausdorff, Proposition~
\ref{Glimm_prop:characterization_of_smooth_orbits} yields that all orbits
are homogeneous spaces so that the assumptions of Glimm's Theorem are 
satisfied. Thus it follows that $G\setminus\Prim\mcA$ is almost Hausdorff.
In particular, $\Prim\mcA$ is a regular $G$-space by 
Proposition~\ref{Glimm_prop:almost_Hausdorff_implies_regular}. Finally 
part~\textit{3.}\ of Propostion~\ref{Glimm_prop:G_prime_ideals} implies 
that there exists a unique $G$-orbit such that $\res(P)=k(G\mdot Q)$.
\end{proof}

\noindent The next result is a consequence of 
Theorem~\ref{Glimm_thm:compact_groups_act_almost_properly} and 
Proposition~\ref{Glimm_prop:G_prime_ideals}.

\begin{prop}
Let $(G,\mcA)$ be a covariance system where $G$ is a compact group 
and $\mcA$ is a $\Cst$-algebra such that $\Prim\mcA$ is almost 
Hausdorff. Then $\Prim\mcA$ is a regular $G$-space and for every 
prime ideal $P$ of $\Cst(G,\mcA)$ there exists a unique orbit
$G\mdot Q$ of $\Prim\mcA$ such that $\res(P)=k(G\mdot Q)$.
\end{prop}
\begin{proof}
Since $G$ is compact and $\Prim\mcA$ is almost Hausdorff, 
Theorem ~\ref{Glimm_thm:compact_groups_act_almost_properly} implies
that $\Prim\mcA$ is a regular $G$-space and that $G\setminus\Prim\mcA$
is almost Hausdorff. Now part~\textit{3.}\ of 
Proposition~\ref{Glimm_prop:G_prime_ideals} yields the desired result.
\end{proof}

\noindent Let $N$ be a closed normal subgroup of a locally compact group $G$. 
Note that $G$ acts on~$N$ by conjugation $n^z:=z^{-1}nz$. The Haar measure 
of~$N$ satisfies $dn^z=\delta(z)dn$ with a continuous homomorphism 
$\delta:N\to\mR_{+}$. Now we see that $a^z(n):=\delta(z^{-1})\,a(n^{z^{-1}})$
defines a continuous action of $G$ on $L^1(N)$ and $\Cst(N)$. This means
that $(G,\Cst(N))$ is a covariance system. We observe that the formulas
\[(a\ast f)(x):=\int_N a(n)f(n^{-1}x)\,dn\quad\text{and}\quad
 (f\ast a)(x):=\int_N f(xn)a(n^{-1})\,dn\]
turn $\Cst(G)$ into a $\Cst(N)$-bimodule. In particular, it holds 
$(ab)\ast f=a\ast(b\ast f)$, $a\ast(f\ast g)=(a\ast f)\ast g$, 
$|a\ast f|_\ast\le |a|_\ast |f|_\ast$, $(a\ast f)^\ast=f^\ast\ast a^\ast$, 
$(f\ast a)^\ast=a^\ast\ast f^\ast$, $(a\ast f)\ast b=a\ast(f\ast b)$ and 
$(f\ast a)\ast g=f\ast(a\ast g)$ for $f,g\in\Cst(G)$ and 
$a,b\in\Cst(N)$. Furthermore
\[f^z(x):=\Delta_G(z^{-1})f(zxz^{-1})\]
defines a continuous action of $G$ on $\Cst(G)$. We have
$(f\ast g)^z=f^z\ast g^z$, $(f^\ast)^z=(f^z)^\ast$ and 
$(a\ast f)^z=a^z\ast f^z$ for $f,g\in\Cst(G)$, $z\in G$ and 
$a\in\Cst(N)$.\\\\
Finally we define
\[\res(P)=\{a\in\Cst(N):a\ast\Cst(G)\subset P\}\]
for ideals $P$ of $\Cst(G)$ and
\[\ext(J)=\langle\,\Cst(G)\ast J\ast\Cst(G)\rangle\,\clos\]
for ideals $J$ of $\Cst(N)$.
\begin{lem}\label{Glimm_lem:induced_ideals_of_CstG}
 Let $J$ be a $G$-invariant ideal of $\Cst(N)$. Then $J\ast\Cst(G)$ is 
 contained in the closure of $\Cst(G)\ast J$, and $\Cst(G)\ast J$ in the 
 closure of $J\ast\Cst(G)$.
\end{lem}
\begin{proof}
 To begin with, we note that there exists a unique surjective homomorphism
 $\Phi:\Cst(G,\Cst(N))\to\Cst(G)$ satisfying 
 $\Phi(f)(x)=\int_N f(xn,n^{-1})\,dn$ for $f\in\mcC_0(G\ltimes N)$. It
 holds $\Phi(f^z)=\Phi(f)^z$ and $\Phi(a\ast f)=a\ast\Phi(f)$.\\
 Let $a\in J$ and $\bar{f}\in\Cst(G)$ be arbitrary. Since $\Phi$ is surjective,
 there is $f\in\Cst(G,\Cst(N))$ such that $\Phi(f)=\bar{f}$. By
 Lemma~\ref{Glimm_lem:induced_ideal} there exist $a_n\in J$ and 
 $f_n\in\Cst(G,\Cst(N))$ such that $\lim f_n\ast a_n=a\ast f$. Applying $\Phi$ 
 we get $\lim\Phi(f_n)\ast a_n=\lim\Phi(f_n\ast a_n)=\Phi(a\ast f)=a\ast\bar{f}$.
 This proves the first inclusion. The second one can be proved similarly.
\end{proof}

\begin{thm}
Let $N$ be a second countable closed normal subgroup of a hereditary 
Lindel\"of locally compact group $G$.  Suppose that $N$ is of type I and 
that all orbits of $\widehat{N}\cong\Prim\Cst(N)$ are locally closed. Then 
for every factor representation $\pi$ of $G$ there exists a unique orbit 
$G\mdot\sigma$ of $\widehat{N}$ such that $\pi\,|\,N$ is weakly equivalent
to $G\mdot\sigma$, i.e.,
\[\ker_{\Cst(N)}\pi=\bigcap_{g\in G}\ker_{\Cst(N)}g\mdot\sigma\;.\]
\end{thm}
\begin{proof}
By Th\'eor\`eme~9.1 of~\cite{Dix3} it follows that $\Cst(N)$ is postliminal. 
By Th\'eor\`eme~4.3.7 and 4.4.5 of~\cite{Dix3}, $\widehat{N}\cong\Prim\Cst(N)$ 
is almost Hausdorff. Furthermore $\Prim\Cst(N)$ is second countable because $N$ 
is second countable. Recall that $G$ acts on $\Cst(N)$ so that $\widehat{N}$ 
becomes a $G$-space in a natural way. All $G$-orbits are assumed to be locally 
closed. By Proposition~\ref{Glimm_prop:characterization_of_smooth_orbits}
all orbits are homogeneous spaces. Now
Theorem~\ref{Glimm_thm:characterization_of_nice_actions} implies that
$G\setminus\widehat{N}$ is almost Hausdorff. Moreover $\widehat{N}$ is a regular 
$G$-space by Proposition~\ref{Glimm_prop:almost_Hausdorff_implies_regular}.\\
Let $\pi$ be a factor representation of $G$. By 
Lemma~\ref{Glimm_lem:factor_representations}, $P:=\ker_{\Cst(G)}\pi$
is prime. The proof of Proposition~\ref{Glimm_prop:G_prime_ideals}, together
with Lemma~\ref{Glimm_lem:induced_ideals_of_CstG}, implies that 
$\res(P)=\ker_{\Cst(N)}\pi$ is $G$-prime and that there exists a unique
orbit $G\mdot\sigma$ of $\widehat{N}$ such that $\res(P)=k(G\mdot\sigma)$.
This proves the claim. 
\end{proof}


\begin{thebibliography}{MMMM}
 \bibitem{Blatt1} R.\ J.\ Blattner, {\sl Group extension representations
 and the structure space.} Pac. J. Math. 15 (1965), pp.\ 1101-1113.
  \bibitem{Dix3} J.\ Dixmier, {\sl Les {$C\sp{\ast} $}-alg\`ebres et leurs
 repr\'esentations.} Deuxi\`eme \'edition. Cahiers Scientifiques, Fasc. XXIX,
 Gauthier-Villars \'Editeur, Paris, 1969.
 \bibitem{Foll} G.\ B.\ Folland, {\sl A course in abstract harmonic analysis.}
 Studies in Advanced Mathematics, CRC Press, Boca Raton, Florida, 1995.
 \bibitem{Glimm1} J.\ Glimm, {\sl Locally compact transformation groups.}
 Trans. Am. Math. Soc. 101 (1961), pp.\ 124-138.
 \bibitem{Green1} P.\ Green, {\sl The local structure of twisted covariance algebras.}
 Acta Math. 140 (1978), no.\ 3-4, pp.\ 191--250.
 \bibitem{Rief2} M.\ A.\ Rieffel, {\sl Unitary representations of group extensions;
 an algebraic approach to the theory of {M}ackey and {B}lattner.}
 Studies in analysis, Adv. in Math. Supl. Stud. 4, pp.\ 43--82,
 Academic Press, New York, 1979.
 \bibitem{Will} Dana P.\ Williams, {\sl Crossed Products of C*-algebras.}
 Mathematical surveys and monographs, American Mathematical Society, Vol.\  134, 2007.
 \end{thebibliography}
\end{document}